\documentclass[a4paper,twoside]{amsart}
\usepackage[utf8]{inputenc}
\usepackage[hidelinks]{hyperref}
\usepackage{graphicx}
\usepackage{arydshln}
\usepackage{fmtcount}
\usepackage{mathtools}
\usepackage{enumitem}
\usepackage{amssymb} 					%for \smallfrown
\usepackage{latexsym} 					%for \smallfrown

\def\Succ{\mathop{\mathrm{Succ}}\nolimits}
\def\ImmSucc{\mathop{\mathrm{ImmSucc}}\nolimits}
\def\Str{\mathop{\mathrm{Str}}\nolimits}
\def\str#1{\mathrm{#1}}

\def\nth#1{$#1$\textsuperscript{th}}

\newcommand{\conc}{^\smallfrown}           %concatenation
\newcommand{\card}[1]{\left| #1 \right|}    %cardinality
\DeclareMathOperator{\val}{val}				%valuation tree

\theoremstyle{definition}
\newtheorem{definition}{Definition}[section]
\newtheorem{example}{Example}
\newtheorem*{remark*}{Remark}
\newtheorem*{claim*}{Claim}
\theoremstyle{remark}
\newtheorem{remark}{Remark}[section]

\theoremstyle{plain}
\newtheorem{theorem}{Theorem}[section]

\newtheorem{observation}[theorem]{Observation}
\newtheorem{lemma}[theorem]{Lemma}

\def\str#1{\mathbf{#1}}

\begin{document}
\bibliographystyle{alpha}
\title{Big Ramsey degrees of 3-uniform hypergraphs are finite}

\author[M. Balko]{Martin Balko}
\address{Department of Applied Mathematics (KAM), Charles University, Ma\-lo\-stransk\'e n\'a\-m\v es\-t\'\i~25, Praha~1, Czech Republic}
\email{balko@kam.mff.cuni.cz}

\author[D. Chodounsk\'y]{David Chodounsk\'y}
\address{Department of Applied Mathematics (KAM), Charles University, Ma\-lo\-stransk\'e n\'a\-m\v es\-t\'\i~25, Praha~1, Czech Republic}
\email{chodounsky@math.cas.cz}

\author[J. Hubi\v cka]{Jan Hubi\v cka}
\address{Department of Applied Mathematics (KAM), Charles University, Ma\-lo\-stransk\'e n\'a\-m\v es\-t\'\i~25, Praha~1, Czech Republic}
\curraddr{}
\email{hubicka@kam.mff.cuni.cz}

\author[M. Kone\v{c}n\'{y}]{Mat\v ej Kone\v{c}n\'{y}}
\address{Department of Applied Mathematics (KAM), Charles University, Ma\-lo\-stransk\'e n\'a\-m\v es\-t\'\i~25, Praha~1, Czech Republic}
\email{matej@kam.mff.cuni.cz}

\author[L. Vena]{Lluis Vena}
\address{Department of Applied Mathematics (KAM), Charles University, Ma\-lo\-stransk\'e n\'a\-m\v est\'\i~25, Praha~1, Czech Republic}
\email{lluis.vena@gmail.com}

\thanks{All authors are supported  by  project  18-13685Y  of  the  Czech  Science Foundation (GA\v CR). First and third author are additionally supported by Center for Foundations of Modern Computer Science (Charles University project UNCE/SCI/004) and by the PRIMUS/17/SCI/3 project of Charles University. Fourth author is supported by the Charles University project GA UK No 378119. This paper is part of a project that has received funding from the European Research Council (ERC) under the European Union’s Horizon 2020 research and innovation programme (grant agreement No 810115).}

\begin{abstract}
	We prove that the universal homogeneous 3-uniform hypergraph has finite big
	Ramsey degrees.  This is the first case where big Ramsey degrees are known to
	be finite for structures in a non-binary language.
	
	Our proof is based on the vector (or product) form of Milliken's Tree Theorem
	and demonstrates a general method to carry existing results on structures in
	binary relational languages to higher arities.
\end{abstract}

\maketitle

\section{Introduction}

Given 3-uniform hypergraphs~$\str{A}$ and~$\str{B}$,
we denote by $\str{B} \choose \str{A}$ the set of all embeddings
from~$\str{A}$ to~$\str{B}$. We write
$\str{C}\longrightarrow {(\str{B})}^\str{A}_{k,\ell}$ to denote the following statement:
\begin{quote}
	For every colouring
	$\chi$ of $\str{C}\choose\str{A}$ with $k$ colours, there exists an embedding
	$f\colon\str{B}\to\str{C}$ such that $\chi$ does not take more than $\ell$ values on
	$f[\str{B}]\choose \str{A}$.
\end{quote}
For a countably infinite structure~$\str{B}$ and its finite induced sub-struc\-ture~$\str{A}$,
the \emph{big Ramsey degree} of $\str{A}$ in $\str{B}$ is the least number
$\ell\in \mathbb \omega + 1$
such that $\str{B}\longrightarrow
	{(\str{B})}^\str{A}_{k,\ell}$ for every $k\in \mathbb \omega$; see~\cite{Kechris2005}.
A countably
infinite structure $\str{B}$ has \emph{finite big Ramsey degrees} if the big
Ramsey degree of $\str{A}$ in $\str{B}$ is finite for every finite substructure $\str{A}$ of $\str{B}$.

A countable hypergraph $\str{A}$ is \textit{(ultra)homogeneous} if every
isomorphism between finite induced sub-hypergraphs extends to an automorphism of
$\str{A}$.  It is well known that there is (up to isomorphism) a unique countable
homogeneous 3-uniform hypergraph $\str{H}$ with the property that every countable
3-uniform hypergraph can be embedded into $\str{H}$, see e.g.~\cite{Macpherson2011}.

Solving a question of Sauer\footnote{Personal communication, 2014.}
we prove the following result, which was announced in~\cite{Hubickabigramsey}.
\begin{theorem}\label{thm:3uniform}
	The universal homogeneous 3-uniform hypergraph $\str{H}$ has finite big Ramsey degrees.
\end{theorem}
Our result is a contribution to the ongoing project of characterising big
Ramsey degrees of homogeneous structures~\cite{Kechris2005}, \cite[Chapter~6]{todorcevic2010introduction}. 
The origin of this project is in the work of Galvin~\cite{Galvin68,Galvin69} who proved that 
the big Ramsey degree of pairs in the order of the rationals, denoted by $(\mathbb Q,\leq)$, is equal to 2. 
Subsequently, Laver in late 1969 proved that in fact $(\mathbb Q,\leq)$ has finite big Ramsey degrees, 
see~\cite[Page~73]{devlin1979},\cite{erdos1974unsolved,laver1984products} and 
Devlin determined the exact values of $\ell$~\cite{devlin1979}, \cite[Theorems~6.22 and 6.23]{todorcevic2010introduction}.
Using Milliken's Tree Theorem (Theorem~\ref{thm:Milliken}) for a
single binary tree his  argument is particularly intuitive: The
vertices of a binary tree can be seen as finite $\{\, 0,1 \,\}$-words and those as
rationals in the range $(0,1)$ written as binary numbers (with additional digit 1 added to the end of each word to avoid ambiguities).
Since this is a dense linear order it follows that $(\mathbb Q,\leq)$ can be embedded to it.
A colouring of finite subsets of $\mathbb Q$ corresponds then to a colouring of
finite subtrees and thus leads to an application of Milliken's Tree Theorem, 
see~\cite[Section 6.3]{todorcevic2010introduction} for further details. 

Similar ideas can be applied to graphs;
here a graph is coded using a binary tree, where $1$ is used to code an edge. In the 
\emph{passing number representation}, a pair of words $w$, $w'$ with $\card{w}\leq \card{w'}$ is
adjacent if $w'_{\card{w}}=1$~\cite[Theorem 6.25]{todorcevic2010introduction}.
Here, $\card{w}$ denotes the length of the word $w$ and $w_i$ is the letter of $w$ on index $i$, where the indices start from $0$.
This representation was used by Sauer~\cite{Sauer2006} and Laflamme, Sauer, Vuksanovic~\cite{Laflamme2006} who, in
2006, characterised big Ramsey degrees of the Rado graph. 
This was refined to unconstrained
structures in binary languages~\cite{Laflamme2006} and additional special
classes~\cite{dobrinen2016rainbow,NVT2009,laflamme2010partition,mavsulovic2020finite}. Milliken's Tree Theorem remained
the key in all results in the area (here we consider Ramsey's Theorem as a special case of Milliken's Tree Theorem for the unary tree).
See also~\cite{dobrinen2019ramsey2} for a recent survey.

A generalization of these results to structures  of
higher arities (such as hypergraphs) and to structures forbidding
non-trivial substructures remained open for over a decade.
Recent connections to topological dynamics~\cite{zucker2017} renewed the interest in the area
and both these problems were solved recently. Using set-theoretic techniques,
Dobrinen~\cite{dobrinen2017universal} proved that the universal homogeneous
triangle free graph has finite big Ramsey degrees and subsequently generalized
this result to graphs omitting a clique $K_k$ for any $k\geq 3$~\cite{dobrinen2019ramsey}.
This was further generalised by Zucker~\cite{zucker2020} to free amalgamation
classes in binary languages.  Hubi\v cka applied parameter spaces and Carlson--Simpson's Theorem to show the finiteness of big Ramsey degrees
for partial orders and metric spaces~\cite{Hubicka2020CS} giving also
a straightforward proof of~\cite{dobrinen2017universal}.

The main goal of this note is to demonstrate a proof technique, which allows
to reprove some of the aforementioned results in the context of relational structures with higher arities.
The proof, for the first time in this area, makes use of the vector (also called
product) form of Milliken's Tree Theorem.  To
our knowledge this may be also the first combinatorial application of Milliken's
Tree Theorem for trees of unbounded branching~\cite{dodos2015}.

A special case of Theorem~\ref{thm:3uniform} (for colouring vertices) also follows
from recent results of Sauer~\cite{Sauer2020} and Coulson, Dobrinen, Patel~\cite{Coulson2020}.

\section{Preliminaries}

Our argument will make use of the vector (or product) form of Milliken's Tree Theorem.
All definitions and results in this section are taken from~\cite{dodos2016}.
Given an integer~$\ell$, we use both the combinatorial notion $[\ell] = \{\, 1,\ldots,\ell \,\}$
and the set-theoretical convention $\ell = \{\, 0,1,\ldots,\ell-1 \,\}$.

A \emph{tree} is a (possibly empty) partially ordered set $(T, <_T)$ such
that, for every $t \in T$, the set $\{\, s \in T : s <_T t \,\}$ is finite and linearly ordered by $<_T$.
All trees considered are finite or countable.
All nonempty trees we consider are \emph{rooted}, that is, they have a unique minimal element called the \emph{root} of the tree.
An element $t\in T$ of a tree $T$ is called a \emph{node} of $T$ and its \emph{level},
denoted by $\card{t}_T$, is the size of the set $\{\,s \in T : s <_T t\,\}$.
Note that the root has level 0.
For $D \subseteq T$, we write $L_T(D) = \{\, \card{t}_T : t \in D \,\}$ for the \emph{level set} of $D$ in $T$.
We use $T(n)$ to denote the set of all nodes of $T$ at level $n$,
and by $T({<}n)$ the set $\{\, t\in T \colon \card{t}_T < n \,\}$.
The \emph{height} of $T$ is the minimal natural number $h$ such that $T(h)=\emptyset$.
If there is no such number $h$, then we say that the height of $T$ is $\omega$.
We denote the height of $T$ by $h(T)$.

Given a tree $T$ and nodes $s, t \in T$ we say that $s$ is a \emph{successor} of $t$ in $T$ if $t \leq_T s$.
The node $s$ is an \emph{immediate successor} of $t$ in $T$  if
$t<_T s$ and there is no $s'\in T$ such that $t<_T s'<_T s$.
We denote the set of all successors of $t$ in $T$ by $\Succ_T (t)$
and the set of immediate successors of $t$ in $T$ by $\ImmSucc_T (t)$.
We say that the tree $T$ is \emph{finitely branching} if $\ImmSucc_T (t)$ is finite for every $t \in T$.

For $s, t \in T$, the \emph{meet} $s\wedge_T t$ of $s$ and $t$ is the largest $s' \in T$ such that $s' \leq_T s$ and $s' \leq_T t$.
A node $t\in T$ is \emph{maximal} in $T$ if it has no successors in $T$.
The tree $T$ is \emph{balanced} if it either has infinite height and no maximal nodes, or all its maximal nodes are in
$T(h-1)$, where $h$ is the height of $T$.

A \emph{subtree} of a tree $T$ is a subset $T'$ of $T$ viewed as a tree
equipped with the induced partial ordering
such that $s \wedge_{T'} t = s \wedge_{T} t$ for each $s,t \in T'$.
Note that our notion of a subtree differs from the standard terminology, since we require the additional condition about preserving meets.

\begin{definition}
	A subtree $S$ of a tree $T$ is a \emph{strong subtree} of $T$
	if either $S$ is empty, or $S$ is nonempty and satisfies the following three conditions.
	\begin{enumerate}
		\item The tree $S$ is rooted and balanced.
		\item Every level of $S$ is a subset of some level of $T$, that is, for every $n < h(S)$ there exists $m \in \omega$ such that $S(n) \subseteq T (m)$.
		\item For every non-maximal node $s \in S$ and every $t \in \ImmSucc_T (s)$ the set
		      $\ImmSucc_S (s) \cap \Succ_T (t)$ is a singleton.
	\end{enumerate}
\end{definition}

\begin{observation}\label{obs-subtree}
	If $E$ is a subtree of a balanced tree $T$, then there exists a strong subtree  $S \supseteq E$ of $T$ such that $L_T(E) = L_T(S)$.\qed
\end{observation}

A \emph{vector tree} (sometimes also called a \emph{product tree}) is a finite sequence $\mathbf T = (T_1, \ldots , T_d)$ of
trees having the same height $h(T_i)$ for all $i \in [d]$.
This common height is the \emph{height} of $\mathbf T$
and is denoted by $h(\mathbf T)$. A vector tree $\mathbf T = (T_1, \ldots , T_d)$ is \emph{balanced}
if the tree $T_i$ is balanced for every $i \in [d]$.

If $\mathbf T = (T_1, \ldots ,T_d)$ is a vector tree, then a \emph{vector subset} of $\mathbf T$ is a sequence
$\mathbf D = (D_1, \ldots , D_d)$ such that $D_i \subseteq T_i$ for every $i \in [d]$.
We say that $\mathbf D$ is \emph{level compatible} if there exists $L \subseteq \omega$ such
that $L_{T_i} (D_i ) = L$ for every $i \in [d]$.
This (unique) set $L$ is denoted by $L_\mathbf T(\mathbf D)$
and is called the \emph{level set} of $\mathbf D$ in $\mathbf T$.

\begin{definition}
\label{def:T1T2}
	Let $\mathbf T = (T_1 , \ldots , T_d )$ be a vector tree.
	A \emph{vector strong subtree} of~$\mathbf T$ is a level compatible vector subset $\mathbf S = (S_1, \ldots , S_d)$ of $\mathbf T$
	such that $S_i$ is a strong subtree of $T_i$ for every $i \in [d]$.
\end{definition}

For every $k \in \omega+1$ with $k \leq h(\mathbf T)$, we use $\Str_k (\mathbf T)$ to denote the set of all vector strong
subtrees of $\mathbf T$ of height $k$.
We also use $\Str_{\leq k}(\mathbf T )$ to
denote the set of all strong subtrees of $\mathbf T$ of height at most $k$.

\begin{theorem}[Milliken~\cite{Milliken1979}]\label{thm:Milliken}
	For every rooted, balanced and finitely branching vector tree $\mathbf T$ of infinite height,
	every nonnegative integer $k$ and every finite colouring of $\Str_k (\mathbf T)$ there is $\mathbf S \in \Str_\omega(\mathbf T)$
	such that the set $\Str_k (\mathbf S)$ is monochromatic.
\end{theorem}

\section{Proof of Theorem~\ref{thm:3uniform}}\label{sec:proof}
Given an integer $n \geq 0$, a \emph{$\{\, 0,1 \,\}$-vector} $\vec{v}$ of length $n$
is a function $\vec{v} \colon n \to 2$.
We write $\card{\vec{v}} = n$ to denote the length of $\vec{v}$ and, 
for $i < \card{\vec{v}}$, we use $v_i$ to denote the \nth{i} coordinate $\vec{v}(i)$ of $\vec{v}$.
In particular, we permit the empty vector and the first coordinate has index 0.
We use the standard vocabulary for matrices. 
An $n\times n$ $\{\, 0,1 \,\}$-matrix $A$ is a function $A \colon n\times n \to 2$. We use $\card{A} = n$ to denote the number of rows (and columns) of $A$.
For $k \leq \card{A}$, we write $A \restriction k$ for the sub-matrix of $A$ with domain $k \times k$.
The \emph{\nth{i} row} of the matrix $A$ is the vector $\vec{v} \colon j \mapsto A_{i, j}$.
The value $A(i,j)$, the entry in the \nth{i} row and the \nth{j} column of $A$ is denoted by $A_{i,j}$.
Note that we start the indexing of entries of $A$ from $0$.
The matrix $A$ is \emph{strictly lower triangular} if $A_{i, j} = 0$ for all $i$ and $j$ with $i \leq j$.

The main idea of the proof of Theorem~\ref{thm:3uniform} is to extend
the passing number representation of graphs to 3-uniform hypergraphs.
This can be done naturally when one understands the passing number
representation in the context of adjacency matrix of a graph as outlined
below.

Consider the universal countable homogeneous graph $\str{R}$ (the Rado graph) and enumerate it by fixing its vertex set $\omega$.
This yields the \emph{asymmetric adjacency matrix} $A$ of $\str{R}$. (Recall that this is an infinite $\{\,0,1\,\}$-matrix with
$A_{j,i}=1$ if and only if $i<j$ and $i$ is adjacent to $j$ in $\str{R}$.)
Assign to every vertex $i\in \omega$ a $\{\,0,1\,\}$-word $w(i)$ which corresponds
to the strictly sub-diagonal part of the \nth{i} row of $A$.  It follows that, for all $i,j\in \omega$ with $i<j$, we have
$|w(i)|=i$ and $w(j)_i=w(j)_{|w(i)|}=1$ if and only if $i$ is
adjacent to $j$ in $\str{R}$. This exactly corresponds to the passing number
representation used to show that big Ramsey degrees of $\str{R}$ are
finite~\cite[Theorem 6.25]{todorcevic2010introduction}, \cite{Sauer2006,Laflamme2006}.

Now consider the countable homogeneous 3-uniform hypergraph $\str{H}$ and put
$H=\omega$. Proceeding analogously as before, one can consider the \emph{asymmetric
	adjacency tensor of $\str{H}$} which is a function $A'\colon n\times n\times
	n\to 2$ defined by $A'(k,j,i)=A'_{k,j,i}=1$ if and only if $i<j<k$ and $i,j,k$ forms an
hyper-edge of $\str{H}$. Now assign to every vertex $i\in\omega$ an $i\times i$
matrix $M(i)$ such that $i<j<k$ forms a hyper-edge of $\str{H}$ if and only if
$M(k)_{j,i}=1$.  This matrix can again be seen as the sub-diagonal part of
a ``slice'' of the adjacency tensor $A'$ since $M(k)_{j,i}=A'_{j,k,i}$ for every
$i,j<k$.

Our proof of Theorem~\ref{thm:3uniform} is based on a refinement of this matrix
representation.   However, in contrast to binary
structures, we need to solve one additional difficulty. The tree of matrices (see
$T_2$ in Definition~\ref{def:trees} and Figure~\ref{fig:T2}) is no longer
uniformly branching and there is no bound on the number of nodes of a strong subtree
of a given height.  This is the main motivation for using a vector tree we
define now.

\begin{definition}\label{def:trees}
	Let $\str{T} = (T_1, T_2)$ be the vector tree, where:
	\begin{enumerate}
		\item The binary tree $(T_1, <_{T_1})$ consists of all finite $\{\, 0,1 \,\}$-vectors ordered by the end-extension. More precisely, we have $\vec{u}\leq_{T_1}\vec{v}$ if $\card{\vec{u}}\leq \card{\vec{v}}$ and $u_i=v_i$ for every $i\in \card{u}$.
		      The root of $T_1$ is the empty vector.
		\item Nodes of the tree $T_2$ are all finite strictly lower triangular (square) $\{\, 0,1 \,\}$-matrices ordered by extension.
		      That is, we have $A \leq_{T_2} B$ if and only if $\card{A} \leq \card{B}$ and $A_{i,j}=B_{i,j}$ for every $i,j \in \card{A}$.
		      The root of $T_2$ is the empty matrix; see Figure~\ref{fig:T2}.
		      \begin{figure}
			      \centering
			      \includegraphics[scale=0.8]{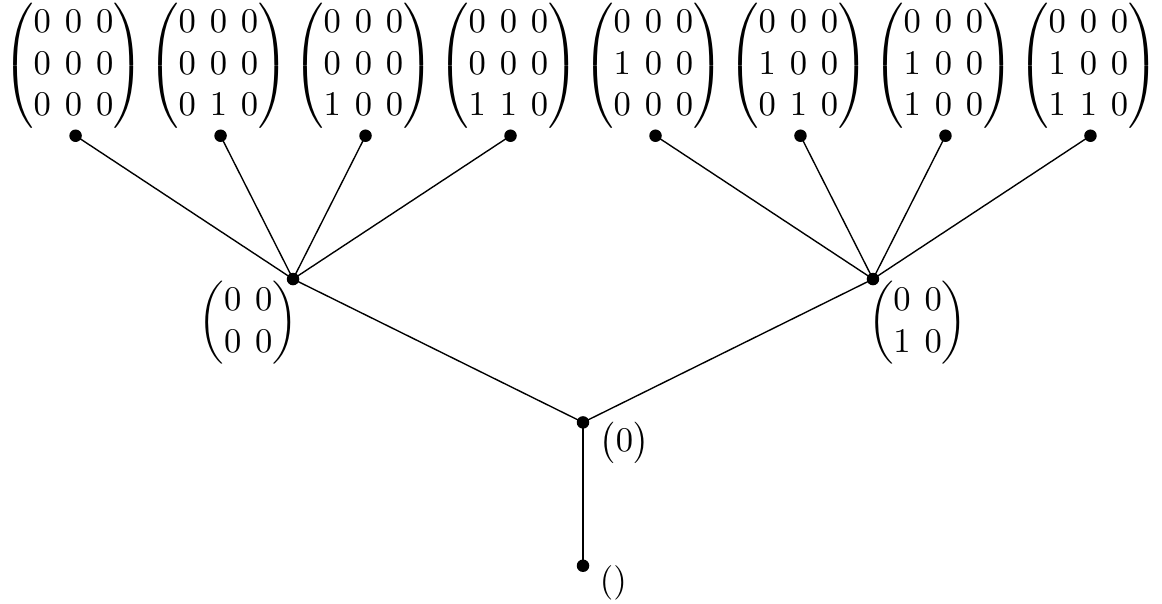}
			      \caption{First 4 levels of the tree $T_2$.}
			      \label{fig:T2}
		      \end{figure}
	\end{enumerate}
\end{definition}
\begin{remark}
	The tree $T_2$ corresponds to the tree of 1-types of $\str{H}$ (see e.g.~\cite{Coulson2020}).
	The tree $T_1$ in our construction has a natural meaning too: while the tree $T_2$ represents vertices and hyper-edges,
	the tree $T_1$ represents the union of all graphs that are created from a 3-uniform hypergraph by fixing a vertex $v$ and considering the graph on the
	same vertex set with edges induced by hyper-edges containing~$v$.
\end{remark}

A key element of the proof is a correspondence between strong vector subtrees of
the vector tree $\str{T}$ and special subtrees of $\str{T}_1$ with 
shape isomorphic to the initial segments of $\str{T}_1$. 

For a matrix $A$ and a vector $\vec{v}$ with $\card{A} = \card{\vec{v}} = n$, the \emph{extension} $A\conc\vec{v}$ of $A$ is the $(n+1)\times(n+1)$ matrix
\[ \left(
	\begin{array}{c;{2pt/2pt}r}
			\mbox{ \LARGE $A$ } & \begin{matrix} 0 \\ \vdotswithin{0}\\0 \end{matrix} \\ \hdashline[2pt/2pt]
			\vec{v}             & 0
		\end{array}
	\right).
\]
More precisely, the matrix $A\conc\vec{v}$ is given by setting
\begin{enumerate}
	\item ${(A\conc\vec{v})}_{i,j} = A_{i,j}$ for all $i, j \in n$,
	\item ${(A\conc\vec{v})}_{n,i} = {v}_i$ for every $i \in n$, and
	\item ${(A\conc\vec{v})}_{j,n} = 0$ for every $j \in n+1$.
\end{enumerate}
Note that if $A$ is a strictly lower triangular matrix, then $A\conc\vec{v}$ is strictly lower triangular as well.

\begin{definition}\label{def:valuation_tree}
	Let $(S_1,S_2)$ be a strong vector subtree of $\str{T}$ of height $k \in \omega +1$.  In other words, $(S_1,S_2) \in \Str_k(\str{T})$.
	The \emph{valuation tree $\val(S_1,S_2)$ corresponding to $(S_1,S_2)$} is a subset of $S_2$ defined by the following
	recursive rules:
	\begin{enumerate}
		\item The root of $\val(S_1,S_2)$ is the root of $S_2$.
		\item If $A \in \val(S_1,S_2)$, $\vec{v} \in S_1(\card{A}_{S_2})$, and
		      $\{\, C \,\} = \ImmSucc_{S_2}(A) \cap \Succ_{T_2}(A\conc \vec{v})$, then $C \in \val(S_1,S_2)$.
		\item There are no other nodes in $\val(S_1,S_2)$.
	\end{enumerate}
	Note that $\val(S_1,S_2)$ is a subtree of $S_2$ and hence also a subtree of~$T_2$.
	Also, the height of $\val(S_1,S_2)$ equals $k$ and the number of nodes of $\val(S_1,S_2)$ depends only on $k$,
	see Lemma~\ref{obs:isom}. 
	
	A tree $T \subseteq T_2$ is a \emph{valuation tree} if $T = \val(S_1,S_2)$ for some $(S_1,S_2) \in \Str_{\leq\omega}(\str{T})$.
\end{definition}
\begin{example}
	\begin{figure}
		\centering
		\includegraphics[scale=0.8]{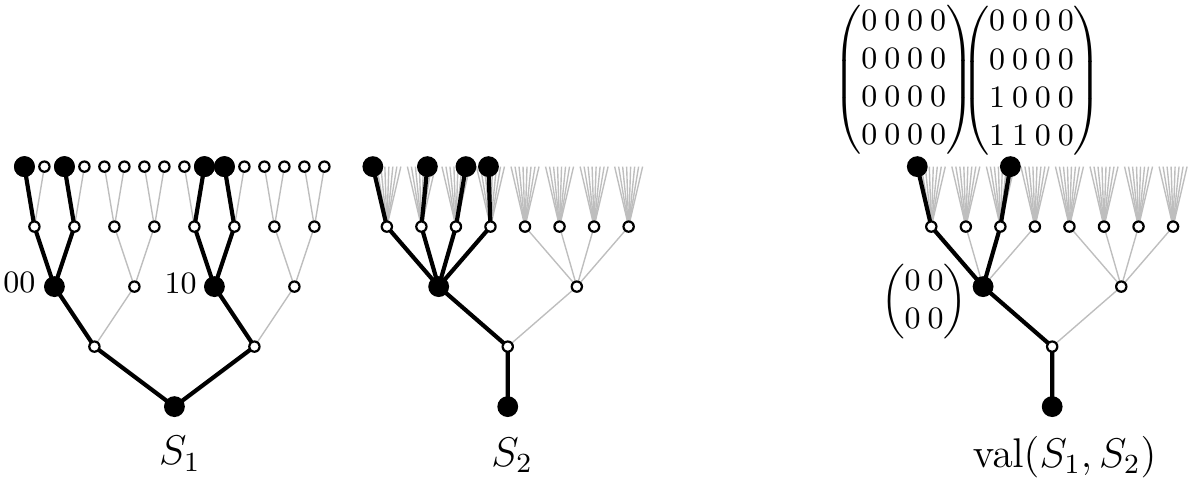}
		\caption{A valuation tree (right) constructed from a vector strong subtree (left).}
		\label{fig:T2b}
	\end{figure}
	See Figure~\ref{fig:T2b} for an example of a valuation tree constructed from a product strong subtree.
\end{example}

For two subtrees $T$ and $T'$ of $T_2$, a function $f \colon T \to T'$ is a \emph{structural isomorphism}
if it is an isomorphism of trees (preserving relative heights of nodes), 
and for every $A, B, C \in T$ with $\card{A}\leq \card{B} < \card{C}$
it also holds that
${f(C)}_{\card{f(B)},\card{f(A)}} = C_{\card{B},\card{A}}$.
\begin{lemma}\label{obs:isom}
	For every $k \in \omega +1$ and every valuation subtree $T$ of $T_2$ of height $k$,
	there exists a unique structural isomorphism $f \colon T_2({<}k) \to T$.
\end{lemma}
\begin{proof}
	For $k\in \omega$ we use induction on $k$.
	There is nothing to prove for $k = 0$. 
	If $k=1$, then the function $f$ mapping the empty matrix to the root of $T$ is the unique structural isomorphism. 
	Assume the induction hypothesis does hold for $k > 0$.
	Let $T = \val(S_1,S_2)$ be of height $k+1$. 
	By the induction hypothesis for the valuation tree $T({<}k)$ there exists 
	a unique structural isomorphism $f \colon T_2({<}k) \to T({<}k)$.  
	Denote by $e \colon k+1 \to \omega$ the increasing enumeration of $L_{T_2}(T)$.	Fix a node $\vec{u} \in T_1(k-1)$. Then there is a unique $\vec{v} \in S_1(k-1)$
	such that $u_{i} =  v_{e(i)}$ for every $i \in k$.
	For every $A \in T_2(k-1)$, there is a unique $C \in \ImmSucc_{S_2}(f(A)) \cap \Succ_{T_2}(f(A)\conc \vec{v})$.
	We extend the map $f$ by declaring $f \colon A\conc\vec{u} \mapsto C$, 
	and we do this for each choice of $\vec{u}$ and $A$.
	It is easy to check that the extended map is a structural isomorphism of $T_2({<}k+1)$ and $T$,
	and that the extension was in fact defined in the unique possible way.
	
	If $k = \omega$, then by the induction hypothesis there are structural isomorphisms
	$f_i \colon T_2({<}i) \to T({<}i)$ for each $i \in \omega$.
	Since these isomorphisms are unique, we get $f_i \subset f_j$ for $i < j$,
	and $f = \bigcup\{\,f_i : i \in \omega \,\}$ is the desired structural isomorphism.
	On the other hand, if $g \colon T_2 \to T$ is a structural isomorphism, 
	then for each $i \in \omega$ the restriction $g \restriction  T_2({<}i) \to T({<}i)$ is a structural isomorhphims
	and due to the induction hypotheses has to be equal to $f_i$, 
	consequently $g = f$.
\end{proof}

\begin{definition}
	\label{def:G}
	Let $\str{G}$ be a 3-uniform hypergraph defined by the following two rules:
	\begin{enumerate}
		\item The vertex set of $\str{G}$ consists of all nodes of $T_2$. In particular, the vertices of $\str{G}$ are square $\{\, 0,1 \,\}$-matrices.
		\item There is a hyperedge $\{\,A,B,C \,\}$ in $\str{G}$ if and only if $A, B, C \in T_2$ are matrices satisfying $\card{A}<\card{B}<\card{C}$ and $C_{\card{B},\card{A}}=1$.
	\end{enumerate}
\end{definition}

Recall that $\str{H}$ denotes the universal countable homogeneous 3-uniform
hypergraph.  Without loss of generality, we assume that the vertex set of
$\str{H}$ is $\omega$.  Let $\varphi \colon \str{H}\to \str{G}$ be an embedding defined by
setting $\varphi(i)=A^i$.
Here, $A^i$ is a $(2i + 1) \times (2i +1)$ matrix such that
if $\{\, j,k,i \,\}$ is a hyper-edge of $\str{H}$ with $j < k < i$, then $A^i_{2k+1,2j}=A^i_{2k+1,2j+1} = 1$
and there are no other non-zero values in $A^i$; see Example~\ref{exam:matrix}.

\begin{example}\label{exam:matrix}
	\setlength\arraycolsep{2pt}
	Assume that $\str{H}$ starts with vertices $\{\, 0,1,2,3 \,\}$ with hyper-edges $\{\, 0,1,2 \,\}$, $\{\, 0,1,3 \,\}$
	and $\{\, 1,2,3 \,\}$.
	Then the corresponding images in $\str{G}$ are:
	\[
		\varphi(0)=
		\begin{pmatrix}
			0
		\end{pmatrix},
		\varphi(1)=
		\begin{pmatrix}
			0 & 0 & 0 \\
			0 & 0 & 0 \\
			0 & 0 & 0 \\
		\end{pmatrix},
		\varphi(2)=
		\begin{pmatrix}
			0 & 0 & 0 & 0 & 0 \\
			0 & 0 & 0 & 0 & 0 \\
			0 & 0 & 0 & 0 & 0 \\
			1 & 1 & 0 & 0 & 0 \\
			0 & 0 & 0 & 0 & 0 \\
		\end{pmatrix},
		\varphi(3)=
		\begin{pmatrix}
			0 & 0 & 0 & 0 & 0 & 0 & 0 \\
			0 & 0 & 0 & 0 & 0 & 0 & 0 \\
			0 & 0 & 0 & 0 & 0 & 0 & 0 \\
			1 & 1 & 0 & 0 & 0 & 0 & 0 \\
			0 & 0 & 0 & 0 & 0 & 0 & 0 \\
			0 & 0 & 1 & 1 & 0 & 0 & 0 \\
			0 & 0 & 0 & 0 & 0 & 0 & 0 \\
		\end{pmatrix}.
	\]
\end{example}

It is easy to check that $\varphi$ is indeed a hypergraph embedding.
We also have the following simple observation.

\begin{observation}\label{obs:matrix}
	For any two matrices $A, B \in \varphi[\str{H}]$,
	\begin{enumerate}
		\item\label{obs:matrix1} all even rows are constant $0$-vectors and so $\card{A \wedge_{T_2} B}$ is odd, and
		\item\label{obs:matrix2} if $\vec{v} \neq \vec{u}$ are two rows of $A$ and $B$, respectively,
		then $\card{\vec{v} \wedge_{T_1} \vec{u}}$ is even.\qed
	\end{enumerate}
\end{observation}

Now, we prove the last auxiliary result that we use in the proof of
Theorem~\ref{thm:3uniform}.  This is a standard step of constructing the
envelope of a set as used by Laver and Milliken~\cite[Section
	6.2]{todorcevic2010introduction}. Here we additionally need to take care of the interactions between the two trees.

\begin{lemma}\label{lem:envelopes}
	For every $k \in \omega$, there exists $R(k) \in \omega$ such that, for every set $S\subset \omega$ of size $k$, there
	exists a valuation tree of height at most $R(k)$ containing all vertices of $\varphi[S]$.
\end{lemma}
\begin{proof}
	Choose an arbitrary natural number $k$, we will show that there is a number $R(k)$ with the desired property.
	To do so, let $S$ be a set of $k$ elements from $\omega$.
	We will construct the required strong vector subtree $(S_1,S_2)$ of $\str{T}$ such that $\val(S_1,S_2)$ contains all matrices from $\varphi[S]$.
	The construction will take a determined number of steps 
	and the upper bound on the height of the constructed tree $(S_1,S_2)$ will thus be a function of $k$.
	
	To achieve this, we define envelopes $E_1$ and $E_2$ in the trees $T_1$ and $T_2$, respectively, by first collecting all necessary matrices in $E_2$,
	then inserting all necessary vectors into $E_1$, which in turn requires adding more matrices to $E_2$.
	The important upshot of our construction is that we can argue that this process promptly terminates and the resulting envelopes are bounded in size.
	
	We proceed in four steps, first defining auxiliary sets $E^0_1 \subseteq T_1$ and $E^0_2 \subseteq T_2$ that will be further extended to $E_1$ and $E_2$, respectively.
	\begin{enumerate}[label=(\roman*)]
		\itemsep0.5em
		\item\label{step1} Let $E^0_2 = \{\, A \wedge_{T_2} B : A,B\in \varphi[S] \,\}\subset T_2$.
		\smallskip\noindent
		This is necessary to obtain a subtree of $T_2$. Observe that $\card{E^0_2}\leq 2k-1$.
		
		\item\label{step2} Let $E^0_1 \subset T_1$ consist of all \nth{\card{B}} rows of $A$ for all $A, B\in E^0_2$, $\card{B} < \card{A}$
		and a constant $0$-vector of length $\max L_{T_2}(E^0_2)$.
		\smallskip\noindent
		This is necessary to obtain a valuation tree that contains all of $E^0_2$. The additional zero vector is added to make the level sets of $E^0_1$ and $E^0_2$ equal. Observe that  $\card{E^0_1}\leq \card{E^0_2}^2+1$.
		
		\item\label{step3} Let $E_1 = \{\, \vec{u} \wedge_{T_1} \vec{v} : \vec{u},\vec{v}\in E^0_1 \,\}$.
		\smallskip\noindent
		This is necessary to obtain a subtree of~$T_1$. Observe that $\card{E_1}\leq 2\card{E^0_1}-1$.
		
		\item\label{step4} Let $E_2$ extend $E^0_2$ by all matrices $A \restriction \card{\vec{v}}$
		where $A \in E^0_2$ and $\vec{v}\in E^0_1$.
		\smallskip\noindent
		This is necessary in order to synchronize levels between both subtrees. Observe that $\card{E_2}\leq \card{E^0_2}({\card{E_1}}+1)$.
	\end{enumerate}
	It follows from step~\ref{step3} that $E_1$ is meet closed in $T_1$ and thus it is a subtree of $T_1$.
	Similarly, step~\ref{step1} implies that $E^0_2$ is a subtree of $T_2$. Thus also $E_2$ is a subtree of~$T_2$, as we did not introduce any new meets in step~\ref{step4}.
	It follows from the upper bounds on $\card{E_1}$ and $\card{E_2}$ that the height of $E_2$ is bounded from above by a function of $k$.
	
	By step~\ref{step4}, the level sets of $E_1$ and $E_2$ are the same, that is, 	$L = L_{T_2}(E_2) = L_{T_1}(E_1)$.
	Now, let $S_1$ be some strong subtree of $T_1$ containing $E_1$ such that $L_{T_1}(S_1) = L$ and let
	$S_2$ be some strong subtree of $T_2$ containing $E_2$ such that $L_{T_2}(S_2) = L$.
	Such trees $S_1$ and $S_2$ exist by Observation~\ref{obs-subtree}.
	
	We claim that $\val(S_1, S_2)$
	contains $\varphi[S]$.
	Choose any matrix $A \in \varphi[S] \subseteq E_2 \subseteq S_2$.
	We prove by induction on the level $\ell \in L$, where $\ell \leq \card{A}$, that $A\restriction \ell \in \val(S_1,S_2)$.
	For the base of the induction, if $\ell$ is the minimal element of $L$, 
	then $A\restriction \ell$ is the root of $S_2$ and hence the root of $\val(S_1,S_2)$.
	
	To prove the induction step, we need to check that if $A\restriction \ell \in S_2$ for $\ell < \card{A}$,
	then the \nth{\ell} row $\vec{v}$ of $A$, is a node of $S_1$.
	Observe that level sets of the constructed sets ``$E$'' are extended only at steps~\ref{step1} and~\ref{step3} of the construction.
	Moreover, all new levels introduced during step~\ref{step1} are odd by part~(\ref{obs:matrix1}) of Observation~\ref{obs:matrix} while levels introduced at step~\ref{step3} are even by part~(\ref{obs:matrix2}) of Observation~\ref{obs:matrix}.
	
	We distinguish two cases based on the parity of the level $\ell$.
	If $\ell$ is odd, then $\ell \in L_{T_2}(E^0_2)$, since all odd levels are introduced only in step~\ref{step1}.
	By step~\ref{step2}, we then have $\vec{v} \in E^0_1$.
	Since $E^0_1 \subseteq E_1 \subseteq S_1$, we have $\vec{v} \in S_1$.

	Otherwise $\ell$ is even.
	Then $\vec{v}$ is a constant $0$-vector by part~(\ref{obs:matrix1}) of Observation~\ref{obs:matrix}. We have $\vec{v} \in S_1$, since one of the maximal nodes of $S_1$ is a constant $0$-vector by step~\ref{step2}.
	
	Altogether, $\val(S_1,S_2)$ contains all matrices from $\varphi[S]$, which finishes the proof.
\end{proof}

We can now proceed with the proof of Theorem~\ref{thm:3uniform}.

\begin{proof}[Proof of Theorem~\ref{thm:3uniform}]
	Fix a finite 3-uniform hypergraph $\str{A}$.
	Recall  that we want to prove that there exists a number $\ell=\ell(\str{A})$ such that for every finite $k$
	$$\str{H}\longrightarrow (\str{H})^\str{A}_{k,\ell}.$$
	That is, for every colouring $\chi^0 \colon {\str{H}\choose\str{A}}\to k$ there is an embedding $g \colon \str{H} \to \str{H}$ such that $\chi^0$ does not take more than $\ell$ values on $\binom{g[\str{H}]}{\str{A}}$.
	
	Consider the 3-uniform hypergraph $\str{G}$ introduced in Definition~\ref{def:G}.
	Since $\str{G}$ is a countable 3-uniform hypergraph, it follows from the properties of $\str{H}$ that there is an embedding $\theta \colon \str{G}\to \str{H}$.
	Consider the colouring $\chi\colon{\str{G}\choose\str{A}} \to k$ obtained by setting
	$\chi\left(\widetilde{\str{A}}\right) = \chi^0\left(\theta\left(\widetilde{\str{A}}\right)\right)$ for every $\widetilde{\str{A}} \in \binom{\str{G}}{\str{A}}$.
	
	Consider the vector tree $\str{T}=(T_1,T_2)$ given by Definition~\ref{def:T1T2}.
	Let $h = R(\card{\str{A}})$ be given by Lemma~\ref{lem:envelopes}.
	Let $\str{G}_h$ be the induced sub-hypergraph of $\str{G}$ on $T_2({<}h)$.
	We enumerate the copies of $\str{A}$ in
	$\str{G}_h\choose\str{A}$ as $\{\, \widetilde{\str{A}}_i : i \in \ell \,\}$ for some $\ell \in \omega$ which will give the upper bound on the big Ramsey degree of $\str{A}$.
	
	By Lemma~\ref{obs:isom}, for every valuation tree $T$ of height $h$,
	there is a structural isomorphism $f_T \colon \str{G}_h \to T$ that is also an isomorphism of
	the corresponding sub-hypergraphs of $\str{G}$.
	Let $\str{S}=(S_1, S_2)$ be a strong subtree of $\str{T}$ of height $h$ and
	consider the structural isomorphism $f = f_{\val(S_1, S_2)} \colon \str{G}_h \to \val(S_1, S_2)$.
	Put \[\bar\chi(\str{S}) = \left\langle\, \chi\left(f\left(\widetilde{\str{A}}_i\right)\right) : i \in n \,\right\rangle,\]
	which is a finite colouring of $\Str_h(\str{T})$.
	By Theorem~\ref{thm:Milliken}, there is an infinite strong subtree of $\str{T}$ monochromatic with respect to $\bar\chi$.
	Let $U$ be its corresponding valuation subtree.
	The structural isomorphism $\psi \colon T_2\to U$ given by Lemma~\ref{obs:isom} is a hypergraph embedding $\psi \colon \str{G}\to \str{G}$.
	Since, by Lemma~\ref{lem:envelopes}, every $\widetilde{\str{A}}\in {\str{G}\choose \str{A}}$
	is contained in at least one valuation subtree of height $h$, we know that $\chi$ takes at most $\ell$ different values on $\psi[\str{G}]\choose\str{A}$.
	Considering the embedding $\varphi \colon \str{H} \to \str{G}$ defined earlier, the image $\theta\left[\psi\left[\varphi[\str{H}]\right]\right]$ is the desired copy $g[\str{H}]$ of $\str{H}$, in which copies of $\str{A}$ have at most $n$ different colours in~$\chi^0$.
\end{proof}

\section{Concluding remarks}
\noindent
{\bf 1.} The construction naturally generalises to $d$-uniform hypergraphs for any $d\geq 2$.
Identifying the underlying set of the $d$-uniform hypergraph with $\omega$ we get a $d$-dimensional adjacency 
$\{\,0,1 \,\}$-tensor. 
We can now consider the hypergraph consisting of $(d-1)$-dimensional `sub-diagonal' $\{\,0,1 \,\}$-tensors 
with the edge relation being defined analogously as in the 3-uniform case. 
The $(d-1)$-dimensional tensors ordered by extension now form a tree $T_{d-1}$.
Sub-hypergraphs isomorphic to $T_{d-1}$ will be again constructed using Milliken's Tree Theorem 
used for the vector tree $\str{T}_d=(T_1,\ldots, T_{d-1})$ and by defining valuation subtrees of $T_{d-1}$.
Nodes of a tree $T_i$ with $1\leq i\leq d-1$ are $\{\, 0,1 \,\}$-tensors of order
$i$ ordered analogously as in the tree $T_2$ used in Section~\ref{sec:proof}. 
The definition of the valuation
tree from Section~\ref{sec:proof} also naturally generalises; given a vector strong subtree $\str{S}=(S_1,\ldots, S_{d-1})$ one first obtains the valuation tree $\val(S_1,S_2)$.
Based on $\val(S_1,S_2)$ and $S_3$ the valuation subtree $\val(S_1,S_2,S_3)$ of $T_3$ can be
constructed in analogy to Definition~\ref{def:valuation_tree}. 
The construction then proceeds similarly for higher orders, defining $\val(S_1,S_2, \ldots, S_i)$ for all $i < d$. 
The final valuation tree $\val(S_1,S_2, \ldots, S_{d-1})$ is the desired subtree of $T_{d-1}$. 
A detailed description of these constructions is going to appear in full generality in~\cite{Balko2020b}.

\medskip

\noindent
{\bf 2.} More generally, structures in a finite relational language with symbols of maximum arity $d$ can be represented
by vector trees $\str{T}_d=(T_1,\ldots, T_{d-1})$. In this case, the nodes of a tree $T_i$ with $1\leq i\leq d-1$ are sequences
of tensors of order $i-d+a$ for every relational symbol of arity $a>i-d$.

\medskip

\noindent
{\bf 3.} We aimed for simplicity in our proof of Theorem~\ref{thm:3uniform}. The bounds obtained in the proof are not optimal.
Structures defined in~\cite{Hubickabigramsey} can be used to produce a more careful
embedding of hypergraphs to $\str{H}$. They describe the order in which the branchings
of the trees $T_1$ and $T_2$ and of the actual vertices appear.
This is also going to appear in~\cite{Balko2020b}.

\section{Acknowledegemnt}
We would like to thank to Stevo Todorcevic for his kind remarks and for helpful discussion regarding
the history and context of this area. We are also grateful to three anonymous referees whose
suggestions improved the presentation of this paper.

\bibliography{ramsey.bib}

\newcommand{\etalchar}[1]{$^{#1}$}
\begin{thebibliography}{LNVTS10}

\bibitem[BCH{\etalchar{+}}19]{Hubickabigramsey}
Martin Balko, David Chodounsk{\' y}, Jan Hubi{\v c}ka, Mat{\v e}j Kone{\v
  c}n{\' y}, and Lluis Vena.
\newblock Big {R}amsey degrees of 3-uniform hypergraphs.
\newblock {\em Acta Mathematica Universitatis Comenianae}, 88(3):415--422,
  2019.

\bibitem[BCH{\etalchar{+}}20]{Balko2020b}
Martin Balko, David Chodounsk{\' y}, Jan Hubi{\v c}ka, Mat{\v e}j Kone{\v
  c}n{\' y}, and Lluis Vena.
\newblock Big {R}amsey degrees of unconstrained relational structures.
\newblock (in preparation), 2020.

\bibitem[CDP20]{Coulson2020}
Rebecca Coulson, Natasha Dobrinen, and Rehana Patel.
\newblock Canonical partitions of relational structures.
\newblock In final stages of preparation, 2020.

\bibitem[Dev79]{devlin1979}
Denis Devlin.
\newblock {\em Some partition theorems and ultrafilters on $\omega$}.
\newblock PhD thesis, Dartmouth College, 1979.

\bibitem[DK16]{dodos2016}
Pandelis Dodos and Vassilis Kanellopoulos.
\newblock {\em Ramsey theory for product spaces}, volume 212.
\newblock American Mathematical Soc., 2016.

\bibitem[DLS16]{dobrinen2016rainbow}
Natasha Dobrinen, Claude Laflamme, and Norbert Sauer.
\newblock Rainbow {R}amsey simple structures.
\newblock {\em Discrete Mathematics}, 339(11):2848--2855, 2016.

\bibitem[Dob19]{dobrinen2019ramsey}
Natasha Dobrinen.
\newblock The {R}amsey theory of {H}enson graphs.
\newblock {\em arXiv:1901.06660, submitted}, 2019.

\bibitem[Dob20a]{dobrinen2017universal}
Natasha Dobrinen.
\newblock The {R}amsey theory of the universal homogeneous triangle-free graph.
\newblock {\em Journal of Mathematical Logic}, page 2050012, 2020.

\bibitem[Dob20b]{dobrinen2019ramsey2}
Natasha Dobrinen.
\newblock Ramsey theory on infinite structures and the method of strong coding
  trees.
\newblock In Adrian Rezus, editor, {\em Contemporary Logic and Computing},
  Landscapes in Logic. College Publications, 2020.

\bibitem[Dod15]{dodos2015}
Pandelis Dodos.
\newblock Some recent results in {R}amsey theory.
\newblock {\em Zbornik Radova}, 17:81--91, 2015.

\bibitem[EH74]{erdos1974unsolved}
Paul Erd{\H{o}}s and Andr{\'a}s Hajnal.
\newblock Unsolved and solved problems in set theory.
\newblock In {\em Proceedings of the Tarski Symposium (Berkeley, Calif., 1971),
  Amer. Math. Soc., Providence}, volume~1, pages 269--287, 1974.

\bibitem[Gal68]{Galvin68}
Fred Galvin.
\newblock Partition theorems for the real line.
\newblock {\em Notices Amer. Math. Soc.}, 15:660, 1968.

\bibitem[Gal69]{Galvin69}
Fred Galvin.
\newblock Errata to ``partition theorems for the real line''.
\newblock {\em Notices Amer. Math. Soc.}, 16:1095, 1969.

\bibitem[Hub20]{Hubicka2020CS}
Jan Hubi{\v{c}}ka.
\newblock Big {R}amsey degrees using parameter spaces.
\newblock {\em arXiv:2009.00967}, 2020.

\bibitem[KPT05]{Kechris2005}
Alexander~S. Kechris, Vladimir~G. Pestov, and Stevo Todor{\v c}evi{\' c}.
\newblock Fra{\"\i}ss{\'e} limits, {R}amsey theory, and topological dynamics of
  automorphism groups.
\newblock {\em Geometric and Functional Analysis}, 15(1):106--189, 2005.

\bibitem[Lav84]{laver1984products}
Richard Laver.
\newblock Products of infinitely many perfect trees.
\newblock {\em Journal of the London Mathematical Society}, 2(3):385--396,
  1984.

\bibitem[LNVTS10]{laflamme2010partition}
Claude Laflamme, Lionel Nguyen Van~Thé, and Norbert~W. Sauer.
\newblock Partition properties of the dense local order and a colored version
  of {M}illiken’s theorem.
\newblock {\em Combinatorica}, 30(1):83--104, 2010.

\bibitem[LSV06]{Laflamme2006}
Claude Laflamme, Norbert~W Sauer, and Vojkan Vuksanovic.
\newblock Canonical partitions of universal structures.
\newblock {\em Combinatorica}, 26(2):183--205, 2006.

\bibitem[Mac11]{Macpherson2011}
Dugald Macpherson.
\newblock A survey of homogeneous structures.
\newblock {\em Discrete Mathematics}, 311(15):1599--1634, 2011.
\newblock Infinite Graphs: Introductions, Connections, Surveys.

\bibitem[Ma{\v{s}}20]{mavsulovic2020finite}
Dragan Ma{\v{s}}ulovi{\'c}.
\newblock Finite big ramsey degrees in universal structures.
\newblock {\em Journal of Combinatorial Theory, Series A}, 170:105137, 2020.

\bibitem[Mil79]{Milliken1979}
Keith~R. Milliken.
\newblock A {R}amsey theorem for trees.
\newblock {\em Journal of Combinatorial Theory, Series A}, 26(3):215--237,
  1979.

\bibitem[NVT09]{NVT2009}
Lionel Nguyen Van~Th{\'e}.
\newblock Ramsey degrees of finite ultrametric spaces, ultrametric {U}rysohn
  spaces and dynamics of their isometry groups.
\newblock {\em European Journal of Combinatorics}, 30(4):934--945, 2009.

\bibitem[Sau06]{Sauer2006}
Norbert~W. Sauer.
\newblock Coloring subgraphs of the {R}ado graph.
\newblock {\em Combinatorica}, 26(2):231--253, 2006.

\bibitem[Sau20]{Sauer2020}
Norbert~W. Sauer.
\newblock Colouring homogeneous structures.
\newblock {\em arXiv:2008.02375}, 2020.

\bibitem[Tod10]{todorcevic2010introduction}
Stevo Todorcevic.
\newblock {\em Introduction to {R}amsey spaces}, volume 174.
\newblock Princeton University Press, 2010.

\bibitem[Zuc19]{zucker2017}
Andy Zucker.
\newblock Big {R}amsey degrees and topological dynamics.
\newblock {\em Groups, Geometry, and Dynamics}, 13(1):235--276, 2019.

\bibitem[Zuc20]{zucker2020}
Andy Zucker.
\newblock A note on big {R}amsey degrees.
\newblock {\em arXiv:2004.13162}, 2020.

\end{thebibliography}

\end{document}